\documentclass[a4paper]{article}
\usepackage{latexsym}
\usepackage{amsmath}
\usepackage{amssymb}
\usepackage{enumerate}
\usepackage{amsthm} 
\usepackage{graphicx} 
\usepackage{amssymb,latexsym,times}
\usepackage{proof}

\newcommand{\C}{\mathcal{C}}
\newcommand{\tuple}[1]{\vec{#1}}
\newcommand {\indep}[3] {#2 ~\bot_{#1}~ #3}
\newcommand {\indepc}[2] {#1 ~\bot~ #2}

\newcommand{\Dom}{\textrm{Dom}}

\newcommand{\Fr}{\textrm{Fr}}

\newcommand{\M}{\mathcal{M}}
\newcommand{\on}{\exists}
\newcommand{\ja}{\wedge}
\newcommand{\tai}{\vee}

\renewcommand{\a}{\alpha}

\def\dep{=\!\!}

\newcommand{\ESOfarity}[1]{{\rm ESO}_f({#1}\mbox{\rm-ary})}

\newcommand{\ESOfvar}[1]{{\rm ESO}_f({#1}\forall)}

\newcommand{\dforall}[1]{\FO(\dep(\ldots))({#1}\forall)}
\newcommand{\ddep}[1]{\FO(\dep(\ldots))({#1}\mbox{\rm-dep})}
\newcommand{\RAM}{{\rm RAM}}
\newcommand{\NTIME}{{\rm NTIME}}
\newcommand{\np}{{\rm NP}}

\newcommand{\PTIME}{{\rm PTIME}}

\newcommand{\LFP}{{\rm LFP}}

\newcommand{\sub}{\subseteq}
\newcommand{\ind}[1]{\FO (\bot_{\rm c})({#1}\mbox{\rm-ind})}

\newcommand{\inc}[1]{\FO (\subseteq)({#1}\mbox{\rm-inc})}
\newcommand{\incforall}[1]{\FO (\subseteq) ({#1}\forall)}
\newcommand{\indforall}[1]{\FO (\bot_{\rm c}) ({#1}\forall)}
\newcommand{\indNRforall}[1]{\FO (\bot) ({#1}\forall)}
\newcommand{\indincforall}[1]{\FO (\bot_{\rm c},\subseteq) ({#1}\forall)}
\newcommand{\indNRincforall}[1]{\FO (\bot,\subseteq) ({#1}\forall)}
\newcommand{\indlogic}{\FO (\bot_{\rm c})}
\newcommand{\inclogic}{\FO (\subseteq)}
\newcommand{\indNRlogic}{\FO (\bot)}
\newcommand{\indNRinclogic}{\FO (\bot,\subseteq)}

\newcommand{\FOCforall}[1]{\FO (\mathcal{C})({#1}\forall)}
\newcommand{\rel}{\textrm{rel}}
\newcommand{\atoms}{\dep(\ldots),\bot_{\rm c},\sub}
\newcommand{\Qf}{\textrm{V}}
\def\dep{=\!\!}
\newcommand{\FO}{{\rm FO}}
\newcommand{\mA}{{\mathfrak A}}
\newcommand{\ESO}{{\rm ESO}}
\newcommand{\PGFP}{{\rm GFP}^+}
\theoremstyle{plain}
\newtheorem{thm}[equation]{Theorem}
\newtheorem{lem}[equation]{Lemma}
\newtheorem{prop}[equation]{Proposition}
\newtheorem{cor}[equation]{Corollary}

\theoremstyle{plain}

\theoremstyle{definition}
\newtheorem{defi}[equation]{Definition}

\begin{document}
\author{Miika Hannula\thanks{Department of Mathematics and Statistics, University of Helsinki, Finland. \texttt{miika.hannula@helsinki.fi}} \and Juha Kontinen\thanks{Department of Mathematics and Statistics, University of Helsinki, Finland. \texttt{juha.kontinen@helsinki.fi}}}
\title{Hierarchies in independence and inclusion logic with strict semantics\thanks{Research supported by grant 264917 of the Academy of Finland}
}

\maketitle

\begin{abstract}
We study the expressive power of fragments of inclusion and independence logic  defined by restricting the number $k$ of universal quantifiers in formulas. Assuming the so-called strict semantics for these logics, we relate these fragments of inclusion and independence logic to  sublogics  $\ESOfvar{k}$ of existential second-order logic, which in turn are known to  capture the  complexity classes $\NTIME_{\RAM}(n^k)$. 
\end{abstract}

\section{Introduction}

In this article we study the expressive power of independence logic  $\indlogic$ \cite{gradel10} and inclusion logic $\inclogic$ \cite{galliani12} under the so-called strict semantics. These logics are variants of dependence logic \cite{vaananen07}  extending first-order logic by dependence atomic formulas 
\begin{equation}\label{da}\dep(x_1,\ldots,x_n)
\end{equation} the meaning of which is that the value of $x_n$ is functionally determined by the values of $x_1,\ldots, x_{n-1}$. 
Independence logic replaces the dependence atoms by independence atoms 
$\vec{y}\bot_{\vec{x}} \vec{z}$,
the intuitive meaning of which is that, with respect to any fixed value of  $ \vec x$, the variables $\vec y$  are independent of the variables $\vec z$.
In inclusion logic dependence atoms are replaced by inclusion atoms
$ \vec{x}\subseteq \vec{y},$
meaning that all the values of $\vec{x}$ appear also as values for $\vec{y}$. We study the expressive power of the syntactic fragments $\indforall{k}$ and $\incforall{k}$ of these logics defined by restricting the number of universal quantifiers in formulas.
We show that, under the strict semantics, the fragments $\indforall{k}$ give rise to an infinite expressivity hierarchy, while it is known that under the so-called lax semantics $\indforall{2} =  \indlogic$  \cite{galhankon13}. For inclusion logic a strict expressivity hierarchy follows from our result showing that
$$\incforall{k}=\NTIME_{\RAM}(n^k).$$

Since the introduction of dependence logic in 2007, the area of dependence logic, i.e.,  team semantics, has evolved into a general framework for logics in which various notions of dependence and independence can be formalized. Dependence logic is known to be equivalent in expressive power with existential second-order logic ($\ESO$) \cite{vaananen07}, but for most of the recent variants of dependence logic the correspondence to $\ESO$ does not hold. Furthermore, the expressive power of these variants is sensitive to the choice between the two versions of the team semantics called the strict and the lax semantics. We briefly mention some previous work on the complexity theoretic aspects of dependence logic and its variants:

\begin{itemize}
\item The extension of dependence logic by so-called intuitionistic implication $\rightarrow$ (introduced in \cite{abramsky09}) increases the expressive power of dependence logic to full second-order logic \cite{yang13}. 
\item The model checking problem of full dependence logic, and many of its variants, was recently shown to be NEXPTIME-complete. Furthermore, for any variant of dependence logic whose atoms are PTIME-computable, the corresponding model checking problem is contained in NEXPTIME  \cite{gradel12}. 
\item The non-classical interpretation of disjunction in dependence logic has the effect that the model checking problem of $\phi_1$ and $\phi_2$, where
\begin{itemize}
\item $\phi_1$ is the formula $\dep(x,y)\vee\dep(u,v)$, and
\item $\phi_2$ is  the formula $ \dep(x,y)\vee\dep(u,v)\vee \dep(u,v)$
\end{itemize}
 is already  NL-complete and NP-complete, respectively \cite{kontinenj13}. 
\item The Satisfiability problem for the two variable fragment of dependence logic was shown to be NEXPTIME-complete in \cite{KontinenKLV11}.
\item Under the lax semantics inclusion logic is equivalent to Positive Greatest Fixed Point Logic ($\PGFP$) and captures $\PTIME$ over finite (ordered) structures \cite{gallhella13}. On the other hand, under the strict semantics, inclusion logic is equivalent to $\ESO$ and hence captures $\np$ \cite{galhankon13}.
\end{itemize}

The starting point of this work are the results of \cite{durand11} and \cite{galhankon13} charting the expressive power of certain natural syntactic fragments of dependence logic and its variants defined using independence and inclusion atoms (See Definition \ref{fragments} for the exact definitions). For a set $\mathcal{C}$ of atoms, we  denote by $\FOCforall{k}$  the sentences of $\FO(\mathcal{C})$ in which at most $k$ variables have been universally quantified, and $\ddep{k}$ denotes those dependence logic sentences in which  dependence atoms of arity at most $k$ may appear (atoms of the form $\dep(x_1,\ldots,x_n)$ satisfying $n\le k+1$). 
The following results were shown in  \cite{durand11}:
\begin{enumerate}
\item \label{arityH}  $\ddep{k}= \ESOfarity{k}$,
\item\label{forallH} $\dforall{k} \le \ESOfvar{k}\le  \dforall{2k}$,
\end{enumerate}
where $\ESOfarity{k}$ is  the fragment  of $\ESO$ in which the quantified functions and relations have arity at most $k$, and 
$\ESOfvar{k}$  consists of $\ESO$-sentences that are in Skolem Normal Form and contain at most  $k$ universal first-order quantifiers. 
Note that \ref{forallH} implies an infinite expressivity hierarchy for the fragments  $\dforall{k}$ by the fact that $\ESOfvar{k} =\NTIME_{\RAM}(n^k)$  \cite{grandjean04}. 

Dependence logic formulas have  the so-called downward closure property which renders the strict and the lax semantics equivalent for dependence logic formulas. The formulas of inclusion logic and independence logic do not have the downward closure property, and hence the two semantics have to be treated separately. In  \cite{galhankon13}, the focus was on fragments of these logics  under the lax semantics. Below the fragments $\ind{k}$ and $\inc{k}$  are defined to contain only those sentences in which independence atoms with at most $k+1$ different variables, and inclusion atoms $\vec{a}\subseteq \vec{b}$ satisfying $|\vec{a}|=|\vec{b}|\leq k$, may appear.
 
\begin{enumerate}[(i)]
\item $ \inc{k}< \ESOfarity{k}=\ind{k},$ 
\item  $\indNRforall{2} =  \indNRlogic$, 
\item  $\indNRincforall{1} = \indNRinclogic$,
\end{enumerate}
where  $\indNRlogic$  is the sublogic of independence logic allowing only so-called pure independence atoms  $\vec{y}\bot \vec{z}$. It is known that  $\indNRlogic = \indlogic$ \cite{vaananen13}. In this article we consider the fragments $\incforall{k}$ and $\indforall{k}$ under the strict semantics. Our findings are comparable to the results of \cite{durand11} (see \ref{forallH}), but the method of proof is different:
\begin{enumerate}
\item\label{main} $ \incforall{k}=\ESOfvar{k}=\NTIME_{\RAM}(n^k)$, 
\item  $\indforall{k} \le \ESOfvar{(k+1)}$,
\item\label{lowerbound} $\ESOfvar{k}\le \indforall{2k} $.
\end{enumerate}
Our results imply an infinite (strict) expressivity hierarchy for the logics $\indforall{k}$ ($\incforall{k}$).

This article is organized as follows. In Section 2 we  review some basic properties and results regarding dependence logic and its variants. In Section 3 we prove a normal form theorem for logics defined using dependence, inclusion, and independence atoms. The main results of the paper are then proved in Section 4. 

\section{Preliminaries}

\subsection{Team Semantics}
In this section we define the essentials of the team semantics of dependence logic. In this paper we consider only formulas in negation normal form, and structures with at least two elements.  The notion of a team is made precise in the next definition.

\begin{defi}

Let $\M$ be a structure with domain $M$, and $V$  a finite set of variables. Then
\begin{itemize}
\item A \emph{team} $X$ over $\M$ with domain $\Dom(X) = V$ is a finite set of assignments from $V$ into  $M$.
\item For a tuple $\tuple v=( v_1,\ldots,v_n)$ of variables from $V$, we denote by $X(\tuple v)$ the $n$-ary relation  $\{s(\tuple v) : s \in X\}$ of $M$, where $s(\tuple v):=( s(v_1),\ldots,s(v_n))$.
\item For a subset $W$ of $V$, we denote by $X \upharpoonright W$  the team obtained by restricting all assignments of $X$ to  $W$.
\item  The set of free variables of a formula $\phi$ is defined analogously as in first-order logic, and is denoted by $\Fr(\phi)$.
\end{itemize}
\end{defi}

We are now ready to define the semantics of dependence logic. We will first define the strict  team semantics and then discuss the ways in which the lax semantics differs from it. We will first define satisfaction for the connectives, quantifiers, and first-order atoms (i.e., first-order formulas). Below  $\M \models_s \alpha$ refers to satisfaction in first-order logic.

\begin{defi}[Strict Semantics]
Let $\M$ be a structure, $X$ and team over $M$, and $\phi$ a formula such that $\Fr(\phi)\subseteq \Dom(X)$. Then $X$ \emph{satisfies} $\phi$ in $\M$, $\M \models_X \phi$, if 
\begin{description}
\item[lit:] For a first-order literal $\alpha$, $\M \models_X \alpha$ if and only if, for all $s \in X$, $\M \models_s \alpha$.
\item[$\vee$:]  $\M \models_X \psi \vee \theta$ if and only if,  there are $Y$ and $Z$ such that $Y \cup Z=X$, $Y\cap Z = \emptyset$, $\M \models_Y \psi$ and $\M \models_Z \theta$.
\item[$\wedge$:] $\M \models_X \psi \wedge \theta$ if and only if, $\M \models_X \psi$ and $\M \models_X \theta$.
\item[$\exists$:]  $\M \models_X \exists v \psi$ if and only if, there exists a function $F : X \rightarrow M$ such that $\M \models_{X[F/v]} \psi$, where $X[F/v] = \{s[F(s)/v] : s \in X\}$.
\item[$\forall$:] $\M \models_X \forall v \psi$ if and only if, $\M \models_{X[M/v]} \psi$, where $X[M/v] = \{s[m/v] : s \in X, m \in M\}$.
\end{description}
A sentence $\phi$ is  said to be \emph{true} in $\M$ (abbreviated $\M \models \phi$) if $\M \models_{\{\emptyset\}} \phi$. Sentences $\phi$ and $\phi'$ are said be equivalent, written $\phi \equiv \phi$, if for all models $\M$, $\M \models \phi \Leftrightarrow \M \models \phi'$.
\end{defi}
In the lax semantics, the semantic rule for disjunction is modified by removing the requirement $Y\cap Z = \emptyset$, and  the clause for the existential quantifier  is replaced by 
\begin{description} 
\item $\M \models_X \exists v \psi$ if and only if, there exists a function $H : X \rightarrow \mathcal P(M) \backslash \{\emptyset\}$ such that $\M \models_{X[H/v]} \psi$, where $X[H/v] = \{s[m/v] : s \in X, m \in H(s)\}$.
\end{description}
The meaning of first-order formulas is invariant under the choice between the two semantics. Furthermore, first-order formulas have  the following strong  Flatness property.
\begin{thm}[Flatness]\label{flatness}Let $\M$ be a structure and $X$ a team of $\M$. Then for a  first order formula $\phi$ the following are equivalent: 
\begin{enumerate}
\item $\M \models_X \phi$, 
\item For all $s \in X$, $\M \models_s \phi$. 
\end{enumerate}
\end{thm}

\subsection{Dependencies in Team Semantics}
For the purposes of this paper, the following atoms are considered:
\begin{defi}
\begin{itemize}
\item Let $\tuple x$ be a tuple of variables and let $y$ be another variable. Then $=\!\!(\tuple x, y)$ is a \emph{dependence atom}, with the semantic rule
\begin{description}
\item $\M \models_X \dep(\tuple x, y)$ if and only if for all $s, s' \in X$, if $s(\tuple x)=s'(\tuple x)$, then $s(y)=s'(y)$;
\end{description}
\item Let $\tuple x$, $\tuple y$, and $\tuple z$ be tuples of variables (not necessarily of the same length). Then $\indep{\tuple x}{\tuple y}{\tuple z}$ is a \emph{conditional independence atom}, with the semantic rule
\begin{description}
\item $\M \models_X \indep{\tuple x}{\tuple y}{\tuple z}$ if and only if for all $s, s' \in X$ such that $s(\tuple x)=s'(\tuple x)$, there exists a $s'' \in X$ such that $s''(\tuple x\tuple y \tuple z)=s(\tuple x\tuple y)s'(\tuple z)$.
\end{description}
Furthermore, we will write $\indepc{\tuple x}{\tuple y}$ as a shorthand for $\indep{\emptyset}{\tuple x}{\tuple y}$, and  call it a \emph{pure independence atom};
\item Let $\tuple x$ and $\tuple y$ be two tuples of variables of the same length. Then $\tuple x \subseteq \tuple y$ 
is an \emph{inclusion atom}, with the semantic rule 
\begin{description}
\item $\M \models_X \tuple x \subseteq \tuple y$ if and only if $X(\tuple x) \subseteq X(\tuple y)$; 
\end{description}
\end{itemize}
\end{defi}

Given a collection $\mathcal C \subseteq \{=\!\!(\ldots), \bot_{\rm c}, \subseteq\}$ of atoms, we will write $\FO(\mathcal C)$ (omitting the set parenthesis of $\mathcal{C}$) for the logic obtained by adding them to the language of first-order logic. With this notation dependence logic, independence logic and inclusion logic are denoted by $\FO (\dep(\ldots))$, $\indlogic$ and $\FO(\subseteq)$, respectively. 
We will also write $\indNRlogic$ for the fragment of independence logic containing only pure independence atoms. It is worth noting that the interpretation of the  atoms is the same in both the  strict and the lax semantics. 

The following proposition formalizes the basic relationship between the two semantics:
\begin{prop}\label{strict-lax}\cite{galliani12}
If $\M \models_X \phi$ in the  strict semantics, then $\M \models_X \phi$ in the lax semantics.
\end{prop}
All formulas of  the above-mentioned logics satisfy the following property (with respect to  both semantics):
\begin{prop}[Empty Team Property]
\label{thm:etp}
For all formulas $\phi \in \FO(=\!\!(\ldots), \bot_{\rm c}, \subseteq)$ and  all structures $\M$, $\M \models_\emptyset \phi$. 
\end{prop}
Furthermore, a fundamental property of all dependence logic formulas is Downward Closure:
\begin{prop}[Downwards Closure]
\label{thm:dc}
For all dependence logic formulas $\phi$ and all  $\M$  and $X$, if $\M \models_X \phi$ then $\M \models_Y \phi$ for all $Y \subseteq X$.
\end{prop}
Downward closure is enough to render the two semantics equivalent: 
\begin{prop}\cite{galliani12}
For all dependence logic formulas $\phi$, models $\M$ and teams $X$, $\M \models_X \phi$ holds under the  strict interpretation if and only if it holds under the  lax interpretation.
\end{prop}
The following  useful  and natural Locality property holds generally only with respect to the lax semantics:
\begin{prop}[Locality]\cite{galliani12}
\label{thm:loc}
Let $\phi$ be a formula of $\FO(=\!\!(\ldots), \bot_{\rm c}, \subseteq)$ whose free variables $\Fr (\phi)$ are contained in $V$. Then, for all models $\M$ and teams $X$,  $\M \models_X \phi$ if and only if $\M \models_{X\upharpoonright V} \phi$, under the lax semantics.
\end{prop}
The failure of locality with respect to the strict semantics makes, e.g., the transformation of formulas into prenex normal form highly non-trivial (See Section 3). 

\subsection{Expressive power and complexity}
We elaborate on some of the results and definitions mentioned in the introduction.

The expressive power of  dependence logic coincides with that of existential second-order logic. In fact the following result holds with respect to both the  strict and the lax semantics:
\[ \ESO= \FO(=\!\!(\ldots), \bot_{\rm c}, \subseteq)= \FO(=\!\!(\ldots))=\indNRlogic .   \]
On the other hand, the expressive power of inclusion logic is sensitive to the choice of the semantics:
\[  \FO( \subseteq)= \ESO   \]
under the strict semantics, and  $\FO( \subseteq)= \PGFP$ (over finite structures   $\FO( \subseteq)= \LFP$)  with respect to the lax semantics. In this article, we take a closer look at the expressive power of the fragments  $\incforall{k}$ and  $\indforall{k}$ under the  strict semantics. As in  \cite{durand11}, we relate  these fragments of inclusion and independence logic to the fragments  $\ESOfvar{k}$ of $\ESO$. Recall that  $\ESOfvar{k}$ contains the sentences of $\ESO$  in Skolem Normal Form
\begin{equation*}
\exists f_1\ldots\exists f_n \forall x_1\ldots \forall x_r \psi,
\end{equation*}
where $r\le k$, and $\psi$ is a quantifier-free formula. It was shown in  \cite{grandjean04}  that
\[ \ESOfvar{k}=\NTIME_{\RAM}(n^k), \]
where $\NTIME_{\RAM}(n^k)$ denotes the family  of classes of structures that can be recognized by a non-deterministic RAM in time $O(n^k)$. By the result of  \cite{cook72},
\[ \NTIME_{\RAM}(n^k)< \NTIME_{\RAM}(n^{k+1})\]
hence relating our logics to these classes gives us a method to show the existence of  expressivity hierarchies.

We end this section by defining the syntactic fragments of logics relevant for this article.
\begin{defi}\label{fragments} Let $\mathcal{C}$ be a subset of $\{\dep(\ldots),\bot_{\rm c},\bot,\subseteq\}$ and let $k \in \mathbb{N}$. Then
\begin{enumerate}
\item $\FO(\mathcal{C}) (k-$dep$)$ is the class of sentences of $\FO(\mathcal{C})$ in which dependence atoms of the form $\dep(\vec{z},y)$, where $\vec{z}$ is of length at most $k$, may appear.
\item $\FO(\mathcal{C}) (k-$ind$)$ is the class of sentences of $\FO(\mathcal{C})$ in which independence atoms of the form $\vec{y}\bot_{\vec{x}} \vec{z}$, where  $\vec{x}\vec{y}\vec{z}$  has at most $k+1$ distinct variables, may appear.
\item $\FO(\mathcal{C}) (k-$inc$)$ is the class of sentences of $\FO(\mathcal{C})$ in which inclusion atoms of the form $\vec{a} \subseteq \vec{b}$, where $\vec{a}$ and $\vec{b}$ are of length at most $k$, may appear.
\item $\FO(\mathcal{C}) (k\forall)$ is the class of sentences of $\FO(\mathcal{C})$ in which every variable is quantified exactly once and at most $k$ universal quantifiers occur. 
\end{enumerate}
\end{defi}

\section{A prenex normal form theorem}
In this section we fix $\C \sub \{\atoms\}$, and present a prenex normal form translation for $\FO(\C)$ sentences. 
We will prove the following normal form theorem.
\begin{thm}\label{NNF}
Let $\phi \in \FO(\mathcal{C})(k\forall)$. Then there is a $\phi'\in \FO(\dep(\ldots ),\mathcal{C})(k\forall)$ which is logically equivalent to $\phi$ and of the form
$$\forall x_1 \ldots \forall x_m \on x_{m+1} \ldots \on x_{m+n} (\chi \ja \theta)$$
where $m \leq k$, $\chi$ is a conjunction of $\{\dep(\ldots ),\mathcal{C}\}$-atoms and $\theta$ is a first-order formula.
\end{thm}
As mentioned previously, Proposition \ref{thm:loc} is the key property used in the prenex normal form translations of \cite{galhankon13, hannula13}, and it does not hold in general in the strict team semantics setting.  To illustrate this, let $\phi:= w \sub x$ and $\psi:= u \sub v \tai w \sub v$, and consider a model $\M = \{0,1,2\}$ and a team $X$ defined as

\begin{center}
    \begin{tabular}{ c | c | c | c |}
      & $u$ & $v$ & $w$ \\ \hline
$s_0$ & $0$ & $1$ & $2$ \\ \hline
$s_1$ & $1$ & $0$ & $1$ \\ \hline
$s_3$ & $2$ & $1$ & $0$ \\ \hline
    \end{tabular}
\end{center}
First we note that since $x$ does not appear free in $\psi$, $\forall x \phi \ja \psi$ is logically equivalent to $\forall x(\phi \ja \psi)$ under the lax semantics \cite{hannula13}. However, this is not the case if we are dealing with the strict semantics. Then $\M \models_X \forall x(\phi \ja \psi)$ but $\M \not\models_X \psi$ (and hence $\M \not\models_X \forall x(\phi \ja \psi)$), and the latter is due to the strict disjunction.
It also easy to construct an analogous example where $\psi$ is an existentially quantified formula in which $x$ does not appear free, and $\forall x(\phi \ja \psi)$ is not logically equivalent to $\forall x \phi \ja \psi$ under the strict semantics. 

The following restricted version of Proposition \ref{thm:loc} in Lemma \ref{fakeloc} is however true for the strict semantics. Lemma \ref{rename} states that we can rename variables in quantifier-free $\FO(\C)$ formulas. The proofs of these lemmas are straightforward inductions on the complexity of the formula, and are thus omitted.

\begin{lem}\label{fakeloc}
Let $\chi$ be a conjunction of $\C$-atoms and $\theta$ a first-order formula. Then for all models $\M$ and teams $X$ and sets $V$ with $\Fr(\chi\ja \theta)\subseteq V\subseteq \Dom(X)$,
$$\M \models_X \chi \ja\theta \Leftrightarrow \M \models_{X\upharpoonright V} \chi \ja \theta.$$
\end{lem}

\begin{lem}\label{rename}
Let $\phi \in \FO(\mathcal{C})$ be quantifier-free formula, and let $\M$  be a model and $X$ a team such that $\Fr(\phi) \sub \Dom(X)$. Then for any $x \in \Dom(X)$ and $y \not\in \Dom(X)$, 
$$\M \models_X \phi \Leftrightarrow \M \models_{X'} \phi'$$
where $\phi'$ is obtained from $\phi$ by replacing all occurrences of $x$ by $y$ and $X'$ is obtained from $X$ by replacing each $s \in X$ by $s'$
 which agrees with $s$ in $\Dom(X)\setminus \{x\}$ and maps $y$ to $s(x)$.
\end{lem}
For the translation, we need the following three definitions. Definition \ref{Qf} introduces a mapping that, given a sentence $\phi$ and its subformula $\psi$, gives us the variables over which $\psi$ (as a subformula of $\phi$) is evaluated. 
 Definition \ref{cond} describes some variables of a $\C$-atom as non-conditional. Definition \ref{rel} introduces an operation which relativizes each $\C$-atom in a formula.
\begin{defi}\label{Qf}
Let $\phi \in \FO(\C)$ be a sentence. For a subformula of $\phi$, the mapping $\Qf_{\phi}$ is defined recursively as follows:
\begin{itemize}
\item $\Qf_{\phi}(\phi)= \emptyset$,
\item if $\psi = \psi_0 \ja \psi_1$ or $\psi = \psi_0 \tai \psi_1$, then $\Qf_{\phi}(\psi_i)= \Qf_{\phi}(\psi)$, for $i=1,2$,
\item if $\psi = \on x \psi_0$ or $\psi = \forall x \psi_0$, then $\Qf_{\phi}(\psi_0) = \Qf_{\phi}(\psi) \cup \{x\}$.
\end{itemize}
\end{defi}
\begin{defi}\label{cond}
For a $\C$-atom $\a$, we say that $x$ is \emph{non-conditional} in $\a$ if
\begin{itemize}
\item $\a$ is $\indep{\tuple x}{\tuple y}{\tuple z}$ and $x$ is listed in $\tuple y \tuple z$,
\item $\a$ is $\dep(x_1, \ldots ,x_n)$ and $x$ is $x_n$,
\item $\a$ is $\tuple x \sub \tuple y$ and $x$ is listed in $\tuple x \tuple y$.
\end{itemize}
\end{defi}
Note that a $\C$-atom $\a$ is satisfied in $X$ if each non-conditional variable in $\a$ has the same constant value in $X$.
\begin{defi}\label{rel}
Let $\tuple x$ be a sequence of variables and $\phi\in \FO(\C)$ a formula. Then we let $\rel_{\tuple x}(\phi)$ be the formula obtained from $\phi$ by substituting each $\C$-atom as follows:
\begin{align*}
\dep(\tuple u, v)   \hspace{3mm} &\mapsto\hspace{3mm}  \dep(\tuple x \tuple u, v),\\
\indep{\tuple u}{\tuple v}{\tuple w} \hspace{3mm} &\mapsto \hspace{3mm} \indep{\tuple x\tuple u}{\tuple v}{\tuple w}, \\
\tuple u \sub \tuple v \hspace{3mm} &\mapsto \hspace{3mm} \tuple x \tuple u \sub \tuple x  \tuple v.
\end{align*}
\end{defi}

The following lemma relates $\rel_{\tuple x}$ to the announcement operator $\delta$ introduced by Galliani in \cite{Galliani13E}. Namely, it states that $\M \models_X \rel_{\tuple x}(\a) \Leftrightarrow \M \models_X \delta \tuple x \a $.

\begin{lem}\label{rel.lemma}
Let $\M$ be a model, $X$ a team and $\a$ a $\C$-atom. Then 
$$\M \models_X \rel_{\tuple x}(\a) \Leftrightarrow \forall \tuple a \in M^{|\tuple x|}( \M \models_{X(\tuple x = \tuple a)} \a)$$
where $X(\tuple x = \tuple a):= \{s \in X \mid s(\tuple x)=\tuple a\}$.
\end{lem}

\begin{proof}
Let $\M$ be a model, $X$ a team and $\a$ a $\C$-atom. Since $\M \models_X \dep(\tuple u ,v) \Leftrightarrow \M \models_X \indep{\tuple u}{v}{v}$, it suffices to prove the claim for independence and inclusion atoms.
\begin{itemize}
\item Assume that $\a = \indep{\tuple u}{\tuple v}{\tuple w}$ when $\rel_{\tuple x}(\a) =\indep{\tuple x\tuple u}{\tuple v}{\tuple w}$, and assume first that  $\M \models_X \indep{\tuple x\tuple u}{\tuple v}{\tuple w}$. Let $\tuple a \in M^{|\tuple x|}$. If $s,s'\in X(\tuple x= \tuple a)$ are such that $s(\tuple u)=s'(\tuple u)$, then by the assumption we find $s''\in X$ which agrees with $s$ in $\tuple x \tuple u \tuple v$ and with $s'$ in $\tuple w$. Since $s''\in X(\tuple x=\tuple a)$, we obtain $ \M \models_{X(\tuple x = \tuple a)} \indep{\tuple u}{\tuple v}{\tuple w}$.

For the other direction, assume that for all $\tuple a \in M^{|\tuple x|}$, $\M \models_{X(\tuple x = \tuple a)}  \indep{\tuple u}{\tuple v}{\tuple w}$, and let $s,s' \in X$ be such that $s(\tuple x \tuple u)=s'(\tuple x \tuple u)$. By the assumption we find $s''\in X(\tuple x= s(\tuple x))$ which agrees with $s$ in $\tuple u\tuple v$ and with $s'$ in $\tuple w$. Since $s''(\tuple x) = s(\tuple x)$, we obtain $\M \models_X \indep{\tuple x\tuple u}{\tuple v}{\tuple w}$.

\item Assume that $\a = \tuple u \sub \tuple v $ when $\rel_{\tuple x}(\a)=\tuple x \tuple u \sub \tuple x  \tuple v$, and assume first that $\M \models_X \tuple x \tuple u \sub \tuple x  \tuple v$. Let $\tuple a \in M^{|\tuple x|}$. If $s \in X(\tuple x= \tuple a)$, then by the assumption we find $s'\in X$ such that $s(\tuple x \tuple u)=s'(\tuple x \tuple v)$. Since then $s' \in X(\tuple x= \tuple a)$, we obtain $\M \models_{X(\tuple x =\tuple a)} \tuple u \sub \tuple v$.

For the other direction, assume that for all $\tuple a \in M^{|\tuple x|}$, $\M \models_{X(\tuple x = \tuple a)} \tuple u \sub \tuple v$, and let $s \in X$. By the assumption we find $s'\in X(\tuple x= s(\tuple x))$ such that $s(\tuple u)=s'(\tuple v)$. Since $s(\tuple x) = s'(\tuple x)$, it follows that $\M \models_X \tuple x \tuple u \sub \tuple x  \tuple v$ which concludes the proof.
\end{itemize}
\end{proof}

For Theorem \ref{NNF}, it now suffices to prove the following lemma. In the following proof we will write $X[F/x]$ for the team $\{s(a/x) \mid s \in X, a= F(s \upharpoonright V)\}$ if $F$ is a function from $X \upharpoonright V$ into $M$. In the case of existential quantification, it will be sometimes useful, and always sufficient, to look for a witness $F: X \upharpoonright V \rightarrow M$, for some $V \sub \Dom(X)$.
\begin{lem}
Let $\phi $ be a $\FO(\mathcal{C})$ sentence  in which every variable is quantified exactly once. Then for any subformula $\psi$ of $\phi$, there is a formula $\psi' \in \FO(\dep(\ldots),\mathcal{C})$ of the form
$$\forall x_1 \ldots \forall x_m \on x_{m+1} \ldots \on x_{m+n} (\chi \ja \theta)$$
where 
\begin{enumerate}
\item $\chi$ is a conjunction of $\{\dep(\ldots),\mathcal{C}\}$-atoms where all non-conditional variables are existentially quantified, and $\theta$ is a quantifier-free first-order formula, 
\item $x_1, \ldots ,x_m$ are universally quantified in $\psi$ and $x_{m+1}, \ldots ,x_{m+n}$ are new or existentially quantified in $\psi$,

\item for all models $\M$ and teams $X$ with $\Dom(X) = \Qf_{\phi}(\psi)$,
$$\M \models_X \psi \Leftrightarrow \M \models_X \psi'.$$
\end{enumerate}
\end{lem}
\begin{proof}
Let $\phi $ be a $\FO(\mathcal{C})$ sentence. We prove the claim by induction on the complexity of the subformula $\psi$ of $\phi$.
\begin{itemize}
\item If $\psi$ is a first-order atomic or negated atomic formula, then we choose $\psi':=\psi$. If $\psi$ is a dependency atom with non-conditional $x_1, \ldots ,x_n$, then we let 
$$\psi':= \on x'_1 \ldots \on x'_n (\psi' \ja \bigwedge_{1 \leq i \leq n} x'_i = x_i)$$
where $\psi'$ is obtained from $\psi$ by replacing each occurence of $x_i$ by $x'_i$, for $1 \leq i \leq n$.
Items 1 and 2 of the claim are now satisfied, and item 3 follows immediately by Lemma \ref{fakeloc}.
\item Assume that $\psi = \forall x \psi_0$. Then by the induction assumption, for $\psi_0$ there is
$$\psi'_0 =\forall x_1 \ldots \forall x_m \on x_{m+1} \ldots \on x_{m+n}(\chi \ja \theta)$$
such that items 1-3 hold. We choose $\psi':= \forall x \psi'_0$. Clearly items 1-2 hold. For item 3, we have by the induction assumption that for all $\M$ and $X$ with $\Dom(X) = \Qf_{\phi}(\psi)$,
$$\M \models_X \psi \Leftrightarrow \M \models_{X[M/x]} \psi_0  \Leftrightarrow \M \models_{X[M/x]} \psi'_0  \Leftrightarrow \M \models_{X} \psi'.$$
For the second equivalence note that $\Qf_{\phi}(\psi_0)=\Dom(X[M/x])$.
\item Assume that $\psi = \on x \psi_0$. Then by the induction assumption, for $\psi_0$ there is
$$\psi'_0 =\forall x_1 \ldots \forall x_m \on x_{m+1} \ldots \on x_{m+n}(\chi \ja \theta)$$
such that items 1-3 hold. We define
\begin{equation*}
\psi':=\forall x_1 \ldots \forall x_m \on x\on x_{m+1} \ldots \on x_{m+n}(\dep(\tuple z ,x) \ja\chi \ja \theta)
\end{equation*}
where $\tuple z$ lists $\Qf_{\phi}(\psi)$. Items 1 and 2 clearly hold, we show that item 3 holds. For this let $\M$ and $X$ be such that $\Dom(X)=\Qf_{\phi}(\psi)$. Analogously to the previous case $\M \models_X \psi \Leftrightarrow \M \models_X \on x \psi'_0$, so it suffices to show that 
$$\M \models_X \on x\psi'_0 \Leftrightarrow \M \models_X \psi'.$$ 
Assume first that $\M \models_X \on x\psi'_0$ when there is a function $F_x : X \rightarrow M$ such that
\begin{equation}\label{eq-1}
\M \models_{X[F_x/x]} \psi'_0.
\end{equation} 
Since no reuse of variables is allowed in $\phi$, variables $x,x_1, \ldots ,x_{m+n}$ are pairwise distinct and not in $\Qf_{\phi}(\psi)$ when existential quantification over them preserves the size of the team. Therefore, and by \eqref{eq-1}, we find $F_i: X[M/x_1]\ldots [M/x_{m}] \rightarrow M$, for $1 \leq i \leq n$, such that
$\M \models_{X'} \chi \ja \theta$, for
$$X':=  X[F_x/x][M/x_1]\ldots [M/x_{m}][F_1/x_{m+1}]\ldots [F_n/x_{m+n}].$$
Also, since $\tuple z$ lists $\Dom(X)$, we have $\M \models_{X'}\dep(\tuple z,x)$.
Now $\M \models_X \psi'$ follows since $X'$ is also of the form
$$X[M/x_1]\ldots [M/x_{m}][F_x/x][F_1/x_{m+1}]\ldots [F_n/x_{m+n}].$$
For the other direction, assume that $\M \models_X \psi'$ when we find $F_x,F_i: X[M/x_1]\ldots [M/x_{m}] \rightarrow M$, for $1 \leq i \leq n$, such that
$\M \models_{X'} \dep(\tuple z,x) \ja \chi \ja \theta$ where
$$X':=  X[M/x_1]\ldots [M/x_{m}][F_x/x][F_1/x_{m+1}]\ldots [F_n/x_{m+n}].$$
For $\M \models_X \on x\psi'_0$, it suffices to note that since $\M \models_{X'} \dep(\tuple z,x)$ and $\tuple z$ lists $\Dom(X)$, we can define $F_x$ already on $X$. This concludes the existential case.
\item Assume that $\psi = \psi_0 \ja \psi_1$. Then by the induction assumption, for $\psi_0$ and $\psi_1$ there are
\begin{align}
\psi'_0 &= \forall x_1 \ldots \forall x_m \on x_{m+1} \ldots \on x_{m+n}(\chi_0 \ja \theta_0),\label{psi0}\\
\psi'_1 &= \forall y_1 \ldots \forall y_k \on y_{k+1} \ldots \on y_{k+l}(\chi_1 \ja \theta_1),\label{psi1}
\end{align}
such that items 1-3 hold. We will construct a $\psi'$ for which the induction claim holds. First note that, for $i=1,2$, $\Qf_{\phi}(\psi) = \Qf_{\phi}(\psi_i)$, and hence by the induction assumption 
\begin{equation}\label{eq0}
\M\models_X \psi_i \Leftrightarrow \M \models_X \psi'_i
\end{equation}
for all $\M$ and $X$ with $\Dom(X) = \Qf_{\phi}(\psi)$. First we assume by symmetry that $k \leq m$. We then let $\chi'_1$ and $\theta'_1$ be obtained from $\chi_1$ and $\theta_1$ by replacing each occurence of $y_i$ by $x_i$, for $1 \leq i \leq k$. By item 2 of the induction assumption and the fact that no reusing of variables is allowed, no quantified variable of $\psi_0$ or $\psi_1$ is in $\Qf_{\phi}(\psi)$, and $\psi_0$ and $\psi_1$ do not share any quantified variables. Therefore by Lemma \ref{rename}, for all $\M$ and $X$ with $\Dom(X) = \Qf_{\phi}(\psi)$,
\begin{equation}\label{eq1}
\M \models_X \psi'_1 \Leftrightarrow \M \models_X \forall x_1 \ldots \forall x_k \on y_{k+1} \ldots \on y_{k+l}(\chi'_1 \ja \theta'_1).
\end{equation}
We then let 
\begin{equation}\label{eq2}
\psi^*_1 :=\forall x_1 \ldots \forall x_m \on y_{k+1} \ldots \on y_{k+l}(\chi^*_1 \ja \theta'_1)
\end{equation}
where
\begin{equation}\label{eq-3}
\chi^*_1 := \chi'_1 \ja \bigwedge_{1 \leq i \leq l} \dep(x_1, \ldots ,x_k,y_i).
\end{equation}
Note that in \eqref{eq2}, $m-k$ universal quantifiers are added, and thus new dependence atoms need to be introduced. Also note that in \eqref{eq-3}, the non-conditional variables $y_i$, for $1 \leq i \leq l$, are existentially quantified.  Now, using Lemma \ref{fakeloc} it is straightforward to check that for all $\M$ and $X$ with $\Dom(X) = \Qf_{\phi}(\psi)$,
\begin{equation}\label{eq3}
\M \models_X \forall x_1 \ldots \forall x_k \on y_{k+1} \ldots \on y_{k+l}(\chi'_1 \ja \theta'_1) \Leftrightarrow \M \models_X \psi^*_1.
\end{equation}
We now let
\begin{equation*}
\psi':=\forall x_1 \ldots \forall x_m \on x_{m+1} \ldots \on x_{m+n} \on y_{k+1} \ldots \on y_{k+l}(\chi_0 \ja \chi^*_1 \ja \theta_0\ja \theta'_1).
\end{equation*}
By the induction assumption, $x_1, \ldots , x_m$ are universally quantified in $\psi$, and $ x_{m+1}, \ldots ,x_{m+n}$ and $ y_{k+1}, \ldots  ,y_{k+l}$ are new or existentially quantified in $\psi$. Hence by \eqref{eq0}, \eqref{eq1} and \eqref{eq3} it suffices to show that $\M \models_X \psi'_0 \ja \psi^*_1 \Leftrightarrow \M \models_X \psi'$ for all $\M$ and $X$ such that $\Dom(X) = \Qf_{\phi}(\psi)$. So let $\M$ and $X$ be of this form, and assume first that $\M \models_X \psi'_0 \ja \psi^*_1$. Then there are functions $F_i,G_j: X[M/x_1]\ldots [M/x_m] \rightarrow M$, for $1 \leq i \leq n$ and $1 \leq j \leq l$, such that
$$\M \models_{X_F} \chi_0 \ja \theta_0$$
and 
$$\M \models_{X_G} \chi^*_1 \ja \theta'_1$$
where
$$X_F:= X[M/x_1]\ldots [M/x_m][F_1/x_{m+1}]\ldots[F_n/x_{m+n}]$$ 
and 
$$X_G:= X[M/x_1]\ldots [M/x_m][G_1/y_{k+1}]\ldots[G_l/y_{k+l}].$$
By Lemma \ref{fakeloc}, 
$$\M \models_{X_{F,G}} \chi_0 \ja \chi^*_1 \ja \theta_0 \ja \theta'_1$$
where
$$X_{F,G}:= X[M/x_1]\ldots [M/x_m][F_1/x_{m+1}]\ldots[F_n/x_{m+n}][G_1/y_{k+1}]\ldots[G_l/y_{k+l}].$$
Therefore, $\M \models_X \psi'$. The other direction is analogous, so we conclude the conjunction case.
\item Assume that $\psi = \psi_0 \tai \psi_1$. Again, we will construct a $\psi'$ for which the induction claim holds. Analogously to the conjunction case, we find $\psi'_0$ and $\psi^*_1$ of the form \eqref{psi0} and \eqref{eq2} and such that for all $\M$ and $X$ with $\Dom(X) =\Qf_{\phi}(\psi)$, $\M \models_X \psi_0 \Leftrightarrow \M \models_X \psi'_0$ and $\M \models_X \psi_1 \Leftrightarrow \M \models_X \psi^*_1$.
We then let
\begin{equation}\label{psi'}
 \psi':=\on a \on b \on c\forall x_1 \ldots \forall x_m\on x_{m+1} \ldots \on x_{m+n} \on y_{k+1} \ldots \on y_{k+l}  \xi
\end{equation}
where
\begin{equation}\label{xi}
 \xi:=\hspace{1mm}\dep(b) \ja\dep(c)\ja \rel_{a}(\chi_0) \ja \rel_{a}(\chi^*_1) \ja b \neq c\ja \big (  (\theta_0 \ja a=b) \tai (\theta'_1 \ja a= c)\big ).
\end{equation}
It suffices to show that for all $\M$ and $X$ with $\Dom(X) =\Qf_{\phi}(\psi)$,
\begin{equation}\label{eq4}
\M \models_X \psi'_0 \tai \psi^*_1 \Leftrightarrow \M \models_X \psi'.
\end{equation}
For this note that, although $\psi'$ is not yet of the right form due to the improper alternation of quantifiers, the existential quantifier block $\on a \on b \on c$ can be moved to the right-hand side of $\forall x_1 \ldots \forall x_m$ by proceeding analogously to the case of the existential quantifier.

For \eqref{eq4}, let $\M$ and $X$ be such that $\Dom(X) =\Qf_{\phi}(\psi)$, and assume first that $\M \models_X \psi'_0 \tai \psi^*_1$. Then we find $Y \cup Z = X$, $Y \cap Z = \emptyset$, such that $\M \models_Y \psi'_0$ and $\M \models_Z \psi^*_1$. Moreover, there are functions $F_i: Y[M/x_1]\ldots [M/x_m] \rightarrow M$, for $1 \leq i \leq n$, and $G_j: Z[M/x_1]\ldots [M/x_m] \rightarrow M$, for $1 \leq j \leq l$, such that
$$\M \models_{Y_F} \chi_0 \ja \theta_0$$
and 
$$\M \models_{Z_G} \chi^*_1 \ja \theta'_1$$
where
$$Y_F:= Y[M/x_1]\ldots [M/x_m][F_1/x_{m+1}]\ldots[F_n/x_{m+n}]$$ 
and 
$$Z_G:= Z[M/x_1]\ldots [M/x_m][G_1/y_{k+1}]\ldots[G_l/y_{k+l}].$$
By Lemma \ref{fakeloc}, we have
\begin{equation}\label{eq5}
\M \models_{Y'} \chi_0 \ja \theta_0 \hspace{1mm}\textrm{  and  }\hspace{1mm}
\M \models_{Z'} \chi^*_1 \ja \theta'_1
\end{equation}
where 
$$Y':= Y[M/x_1]\ldots [M/x_m][0/a][0/b][1/c][F_1/x_{m+1}]\ldots[F_n/x_{m+n}][0/y_{k+1}]\ldots[0/y_{k+l}]$$ 
and 
$$Z':= Z[M/x_1]\ldots [M/x_m][1/a][0/b][1/c][0/x_{m+1}]\ldots[0/x_{m+n}][G_1/y_{k+1}]\ldots[G_l/y_{k+l}],$$
for some distinct constant functions $0$ and $1$. For $\M \models_X \psi'$, it now suffices to show that $\M \models_{X'} \xi$ where $X':= Y'\cup Z'$. For the first-order part and dependence atoms $\dep(b)$ and $\dep(c)$, we have by the construction that 
$$\M \models_{X'}\dep(b) \ja\dep(c)\ja  b \neq c\ja (\theta_0 \ja a=b) \tai (\theta'_1 \ja a= c).$$
For the conjunction of relativized $\{\dep(\ldots),\C \}$-atoms $\rel_{a}(\chi_0) \ja \rel_{a}(\chi^*_1)$, first note that $\M \models_{Y'} \chi^*_1$ and $\M \models_{Z'}\chi_0$ because all the non-conditional variables that appear in $\chi^*_1$ and $\chi_0$ are mapped to $0$ in $Y'$ and $Z'$, respectively. By this and \eqref{eq5}, $\M \models_{Y'} \chi_0 \ja \chi^*_1$ and $\M \models_{Z'}\chi_0 \ja \chi^*_1$. Since $X'(a= 0)= Y'$, $X'(a=1)= Z'$, and $X'(a=x)= \emptyset$ for any other $x \in M$, we have by Lemma \ref{rel.lemma} that $\M \models_{X'} \rel_{a}(\chi_0) \ja \rel_{a}(\chi^*_1)$. Hence $\M \models_X \psi'$ which concludes the only if-part of \eqref{eq4}.

For the other direction, assume that $\M \models_X \psi'$. We show that $\M \models_X \psi'_0\tai \psi^*_1$ where $\psi'_0$ and $\psi^*_1$ are of the form \eqref{psi0} and \eqref{eq2}. By the assumption there are functions $H_d: X \rightarrow M$, for $d\in\{a,b,c\}$, and $F_i,G_j:  X[M/x_1]\ldots [M/x_m] \rightarrow M$, for $1 \leq i \leq n$ and $1 \leq j \leq l$, such that
\begin{equation}\label{eq6}
\M \models_{X'}  \dep(b) \ja\dep(c)\ja \rel_{a}(\chi_0) \ja \rel_{a}(\chi'_1) \ja b \neq c\ja \big (  (\theta_0 \ja a=b) \tai (\theta'_1 \ja a= c)\big ),
\end{equation}
for $X'$ defined as
$$X[H_a/a][H_b/b][H_c/c][M/x_1]\ldots [M/x_m][F_1/x_{m+1}]\ldots [F_n/x_{m+n}][G_1/y_{k+1}]\ldots [G_l/y_{k+l}].$$  
First note that by \eqref{eq6}, $H_b$ and $H_c$ are distinct constant functions, say $0$ and $1$, respectively. We then let $Y:= \{s \in X \mid H_a(s) = 0\}$ and $Z:= \{s\in X \mid H_a(s) = 1\}$ when $Y\cup Z = X$ and $Y \cap Z = \emptyset$. It suffices to show that $\M \models_{Y'} \theta_0\ja \chi_0$ and $\M \models_{Z'} \theta'_1\ja \chi'_1$ where
$$Y':=Y[M/x_1]\ldots [M/x_m][F_1/x_{m+1}]\ldots [F_n/x_{m+n}]$$  
and
$$Z':=Z[M/x_1]\ldots [M/x_m][G_1/y_{k+1}]\ldots [G_l/y_{k+l}].$$ 
For $\M \models_{Y'} \theta_0\ja \chi_0$, since
$$Y'= X'(a=0) \upharpoonright (\Dom(X) \cup\{x_1, \ldots ,x_{m+n}\})$$ 
and the variables $a,b,c,y_{k+1}, \ldots ,y_{k+l}$ do not appear in $\theta_0 \ja \chi_0$, it suffices to show by Lemma \ref{fakeloc}  that $\M \models_{X'(a=0)} \theta_0 \ja \chi_0$. For this, first note that by \eqref{eq6} and Theorem \ref{flatness}, for each $s\in X'(a=0)$ we have $s(a) \neq s(c)$ when it follows that $\M \models_s \theta_0$. Hence by Theorem \ref{flatness}, $\M \models_{X'(a=0)} \theta_0$. $\M \models_{X'(a=0)} \chi_0$ follows from \eqref{eq6} and Lemma \ref{rel.lemma}. Therefore $\M \models_{X'(a=0)} \theta_0 \ja \chi_0$ when $\M \models_{Y'} \theta_0\ja \chi_0$. Analogously we obtain $\M \models_{Z'} \chi^*_1 \ja \theta'_1$. This concludes the if-part of \eqref{eq4} and thus the conjunction case and the proof.
\end{itemize}
\end{proof}

\section{$\FO(\dep(\ldots), \bot_{\rm c}, \subseteq)(k\forall) $ as a fragment of  $\ESO$}
\subsection{A translation into $\ESO$}
In this section we will show that $\FO(\dep(\ldots), \bot_{\rm c}, \subseteq)(k\forall)  \leq \ESOfvar{(k+1)}$. Moreover, it  will be shown that $\FO(\dep(\ldots), \subseteq) (k\forall) \leq \ESOfvar{k}$. For the first one, we need the following lemma.
\begin{lem}\label{atomcollapse}
Let $\indep{\tuple x}{\tuple y v}{\tuple z}$ be an independence atom. Then for all models $\M$ and teams $X$,
\begin{equation}\label{contraction}
\M \models_X  \indep{\tuple x}{\tuple y v}{\tuple z} \Leftrightarrow \M \models_X  \indep{\tuple x v}{\tuple y }{\tuple z}  \ja   \indep{\tuple x}{v }{\tuple z} 
\end{equation}
\end{lem}
\begin{proof}
The direction from left to right is straightforward. We prove the converse direction. For this, assume that $\M \models_X  \indep{\tuple x v}{\tuple y }{\tuple z}  \ja   \indep{\tuple x}{v}{\tuple z} $, and let $s,s'\in X $ be such that $s(\tuple x) = s'(\tuple x)$. Then by $\M \models_X  \indep{\tuple x}{v }{\tuple z} $ there is $s_0\in X$ such that $s_0(\tuple x v \tuple z)=s(\tuple x v)s'(\tuple z)$. Since $s(\tuple x v)=s_0(\tuple x v)$, by $\M \models_X  \indep{\tuple x v}{\tuple y }{\tuple z} $ there is $s_1 \in X$ such that $s_1(\tuple x v \tuple y \tuple z)=s(\tuple x v \tuple y)s_0(\tuple z) = s(\tuple x v \tuple y)s'(\tuple z)$.
\end{proof}
Interestingly, the right-to-left implication in \eqref{contraction} corresponds to the so-called Semi-Graphoid axiom \emph{Contraction} going back to  \cite{Dawid1979} where it is listed as one of the  elementary properties satisfied by (statistical) conditional independence.

\begin{prop}\label{t2} For any sentence $\psi\in \FO(\dep(\ldots ),\bot_{\rm c},\subseteq)(k\forall)$ there is a sentence $\tau_{\psi} \in \ESOfvar{(k+1)}$  such that, for all $\mA$:

\[ \mA\models \psi \Leftrightarrow \mA\models \tau_{\psi}. \]
Furthermore, for $\psi\in \FO(\dep(\ldots),\subseteq)(k\forall)$, it holds that  $\tau_{\psi} \in \ESO_f(k\forall)$. 
\end{prop}
\begin{proof} By Lemma \ref{atomcollapse} we may assume that each independence atom that appears in $\psi$ is of the form $\indep{\tuple z}{u}{w}$. We may also assume that $\psi$ is in $\forall\exists$-form:
\begin{equation}\label{blah}
\forall x_1\ldots \forall x_k\exists y_1\ldots \exists y_m\chi ,    
\end{equation}
where $\chi$ is a quantifier-free formula. For inclusion logic sentences this normal form can be assumed if $\chi$ is allowed
to contain also dependence atoms.
The idea is that a team $X$ that arises by evaluating the quantifier prefix of $\psi $ can be represented by a  $k$-ary relation $S$ and functions $f_j$ encoding the values the variables $y_j$ has in the team $X$. In particular, the  variables $y_i$ will be translated by terms $f_i(x_1,\ldots,x_k)$. We will first show how to construct from the quantifier-free formula $\chi\in \FO(\dep(\ldots ),\bot_{\rm c},\subseteq)$ a $\ESOfvar{(k+1)}$ formula $\tau_{\chi}$ such that for all models $ \M$ and teams $X$,
$$\M \models_X \chi \Leftrightarrow (\M,S,\tuple f) \models \tau_{\chi}$$
where $S$ and $\vec{f}$ encode $X$ as above. The translation is defined as follows.
\begin{enumerate}

\item Let $\psi$ be a first-order atomic or negated atomic formula $\a(\tuple u)$. The formula $\tau_{\psi}(S,\vec{f})$ is now defined as \[ \forall x_1\ldots \forall x_k ( S(x_1,\ldots,x_k)\rightarrow \a(\tuple w)) \]
where $\tuple{w}$ arises from $\vec{u}$ by replacing the variables $y_j$ by $f_j(x_1,\ldots,x_k)$.







\item Let $\psi$ be of the form $\dep(z_1,\ldots,z_m)$. The formula $\tau_{\psi}(S,\vec{f})$ is now  defined as
\begin{eqnarray*}
&&\exists f \forall x_1\ldots\forall x_k (S(x_1,\ldots,x_k) \rightarrow z_m=f(t_1,...,t_{m-1}))
\end{eqnarray*}
\noindent where $t_s=x_j$ if $z_s=x_j$,  and $t_s=f_j(x_1,\ldots,x_k)$ if $z_s=y_j$. 

\item Let $\psi$ be of the form $\vec{z}\subseteq \vec{w}$. The formula $\tau_{\psi}(S,\vec{f})$ is now  defined as
\[ \forall x_1\ldots \forall x_k \exists x_{k+1}\ldots \exists x_{2k}(S(x_1,\ldots,x_k) \rightarrow ( S(x_{k+1},\ldots,x_{2k})\wedge \vec{t}= \vec{s})) \]
where $\vec{t}$ arises from $\vec{z}$ by replacing the variables $y_j$ by $f_j(x_1,\ldots,x_k)$, and 
  $\vec{s}$ arises from $\vec{w}$ by replacing the variables $x_i$ by $x_{k+i}$ and $y_j$ by $f_j(x_{k+1}\ldots,x_{2k})$.  

\item Let $\psi$ be of the form $z\bot_{\vec{u}} w$. The formula $\tau_{\psi}(S,\vec{f})$ is now  defined as
\begin{eqnarray*}
&& \on S_1 \on S_2 \forall x_1\ldots \forall x_k \forall x'\exists x_{k+1}\ldots\exists x_{2k}\Big (S(x_1,\ldots,x_k)\rightarrow \big (S_1(\vec{u} z) \ja S_2(\vec{u} w) \wedge (S_2(\vec{u} x') \rightarrow  \\
&& (S(x_{k+1},...,x_{2k})\wedge \vec{u}_1=\vec{u}_2\wedge
z_1=z_2\wedge x'= w_2)\big ) \Big )
\end{eqnarray*}
where $\vec{u}_1$ and $z_1$ are defined by replacing the variables $y_j$ by $f_j(x_1,\ldots,x_k)$. In $\vec{u}_2$,  $w_2$ and $z_2$ the variables $x_j$ are  replaced by $x_{k+j}$ and the variables $y_j$ by   $f_j(x_{k+1},\ldots,x_{2k})$.

\item Let $\psi$ be of the form $\chi\vee \theta$. Now $\tau_{\psi}(S,\vec{f})$ is defined as
\begin{eqnarray*}
&& \exists S_1\exists S_2 \Big(\tau_{\chi}(S_1,\vec{f})\wedge\tau_{\theta}(S_2,\vec{f})\wedge \forall x_1\ldots\forall x_k \big (S(x_1,\ldots,x_k)\rightarrow  \\ &&  [S_1(x_1,\ldots,x_k)\vee S_2(x_1,\ldots,x_k)\wedge  \neg (S_1(x_1,\ldots,x_k)\wedge S_2(x_1,\ldots,x_k))] \big )\Big ).
\end{eqnarray*}

\item Let $\psi$ be of the form $\chi\wedge \theta$. Now $\tau_{\psi}(S,\vec{f})$ is defined as $\tau_{\chi}(S,\vec{f})\wedge \tau_{\theta}(S,\vec{f})$.


\end{enumerate}
We have given a compositional translation for a quantifier-free $\FO(\atoms)$-formula $\chi$.
An easy induction shows that this formula  can be transformed to the form
\[ \exists f_1\ldots\exists f_m \forall x_1\ldots \forall x_{k} \forall x' \theta,     \]
where $\theta$ is a quantifier-free formula and where the functions $f_i$ may have arity greater than $k$.
The sentence $\psi$ in \eqref{blah} is now equivalent to a $\ESOfvar{(k+1)}$ formula $\tau_{\psi}$ of the form
\[ \exists S\exists f_1\ldots f_{m+n}\forall x_1\ldots \forall x_k \forall x' \big (S(x_1,\ldots,x_k)\wedge \theta  \big )  \]

\end{proof}
\subsection{Capturing $\NTIME_{\RAM}(n^k)$ with $\indincforall{k}$}
Let us first note that the following theorem follows immediately from $\ESOfvar{k} \leq \dforall{2k}$ \cite{durand11} and the fact that a dependence atom $\dep(\tuple x ,y)$ can be expressed as the independence atom $\indep{\tuple x}{y}{y}$.
\begin{thm}\label{leqind}
$\ESOfvar{k} \leq \indforall{2k}$.
\end{thm}
Next we will consider inclusion logic.
\begin{thm}\label{leqinc}
$\ESOfvar{k} \leq \incforall{k}$.
\end{thm}

\begin{proof}
Assume that $\phi$ is a sentence of the form
$$\on f_1 \ldots \on f_n \forall x_1 \ldots \forall x_k \psi$$
where $\psi$ is quantifier-free and first-order. Let $1 \leq i \leq n$. Since $\ESOfvar{k} = \ESO_f (k\forall,k\textrm{-ary})$ \cite{grandjean04}, we may assume that $f_i$ is of arity $k$. By the normal form given in Proposition 3.6. in \cite{durand11}, we may assume that each occurrence of $f_i$ in $\psi$ is of the form
\begin{itemize}
\item  $f_i(x_1, \ldots ,x_k)$ or
\item $f_i(f_{l_1}(\vec{x}), \ldots ,f_{l_k}(\vec{x}))$ where $f,f_{l_1}, \ldots ,f_{l_k}$ are pairwise distinct existentially quantified funtion symbols.
\end{itemize}
Also by Proposition 3.6. we may assume that $f_i$ cannot appear both as an inner and an outer function symbol even in different composed terms, and $f_i$ has at least one occurrence of the form $f_i(x_1, \ldots ,x_k)$ in $\psi$. We will now translate $\phi$ to a sentence in $\incforall{k}$.

For $1 \leq i \leq n$, let $c_i$ be the number of composed terms that appear in $\psi$ and where $f_i$ is the outermost function symbol. Then the occurences of $f_i$ in $\psi$ are of the form
$$t_{i}:=f_i(x_1, \ldots ,x_k),$$
and
$$u_{ij} := f_i(t_{l_{1}}, \ldots ,t_{l_{k}}),\textrm{ for } j=1, \ldots ,c_i,$$
where,  for $1 \leq m \leq k$,  $1 \leq l_m \leq n$ and $c_{l_m}=0$. Note that the value of $l_1, \ldots ,l_k$ depends also on the value of $i$ and $j$. This is not written down in the notation since the value of $i$ and $j$ will be always clear from the context.

We now define $\phi^* \in \incforall{k}$ as

\begin{align*}&\forall x_1 \ldots \forall x_k \on (y_{i})_{1\leq i \leq n}\on (z_{ij})_{\substack{1\leq i \leq n\\1 \leq j \leq c_i}}\big (\psi^* \ja \bigwedge_{\substack{1 \leq i \leq n\\1 \leq j \leq c_i}}  y_{l_{1}} \ldots y_{l_{k}} z_{ij} \subseteq x_1 \ldots x_k y_i \big )
\end{align*}

where $\psi^*$ is obtained from $\psi$ by replacing each term $t_i$ (or $u_{ij}$) with variable $y_i$ (or $z_{ij}$). Let us show that the equivalence holds.



\textit{Only-if part}. Assume that $\M \models \phi$. Then there is $\M^* := (\M, f_1^{\M^*}, \ldots ,f_n^{\M^*})$ such that 
$$\M^* \models \forall x_1 \ldots \forall x_n \psi.$$
Then by Theorem \ref{flatness}
$$\M^* \models_X\psi$$
where $X:= \{\emptyset\}[M/x_1]\ldots [M/x_k]$. Let us extend each $s \in X$ with $y_i \mapsto t_i^{\M^*}\left < s \right >$ and $z_{ij} \mapsto u_{ij}^{\M^*}\left < s\right >$,
for $1 \leq i \leq n$ and $1 \leq j \leq c_i$. Let $X'$ consist of these extended assignments. Clearly $\M \models_{X'} \psi^*$. For 
$$\M \models_{X'}  \bigwedge_{\substack{1 \leq i \leq n\\1 \leq j \leq c_i}}  y_{l_{1}} \ldots y_{l_{k}} z_{ij} \subseteq x_1 \ldots x_k y_i  ,$$
let $1 \leq i \leq n$, $1 \leq j\leq c_i$ and $s \in X'$. We show that there is $s'\in X'$ such that
$$s(y_{l_{1}}, \ldots ,y_{l_{k}} ,z_{ij}) =s'( x_1, \ldots ,x_k, y_i ).$$
We let $s' \in X'$ be such that
$$s(y_{l_{1}}, \ldots ,y_{l_{k}})=s'( x_1, \ldots ,x_k).$$
Note that
$$s(y_{l_{1}}, \ldots ,y_{l_{k}})=(t_{l_{1}}^{\M^*}\left < s\right >, \ldots ,t_{l_{k}}^{\M^*}\left < s\right >)$$
by the construction. Then
\begin{align*}
s(z_{ij})& =u_{ij}^{\M^*}\left < s \right >= f_i^{\M^*} (t_{l_{1}}^{\M^*}\left <s\right >, \ldots ,t_{l_{k}}^{\M^*}\left <s\right > )\\
& = f_i^{\M^*} (s'(x_1), \ldots ,s'(x_k))= t_i^{\M^*}\left <s'\right> =s'(y_{i} )
\end{align*}
which shows the claim and concludes the \textit{only-if part}.

\textit{If-part}. Assume that $\M \models \phi^*$. Then there are functions $F_i,G_{ij}: \{\emptyset\}[M/x_1]\ldots [M/x_k] \rightarrow M$, for $1 \leq i \leq n$ and $1 \leq j \leq c_i$, such that $\M \models_X \psi^*$ where
$$X:=\{\emptyset\}[M/x_1]\ldots [M/x_k][F_i/y_i]_{1\leq i \leq n}[G_{ij}/z_{ij}]_{\substack{1\leq i\leq n \\2 \leq j \leq c_i}}.$$ 
Let $\M^*:=(\M, f_1^{\M^*}, \ldots ,f_n^{\M^*})$ be such that, for $1 \leq i \leq n$ and $s \in X$,
\begin{equation}\label{pla}
f_i^{\M^*} (s(x_1), \ldots ,s(x_k)) = s(y_i).
\end{equation}
Because of strict semantics, $f^{\M^*}$ is well defined. We will show that 
$$\M^* \models \forall x_1 \ldots \forall x_k \psi.$$
For this, let $s \in X$ when $\M \models_s \psi^*$. We will show that $\M^* \models_{s'} \psi$ where $s':=s\upharpoonright \{x_1, \ldots ,x_k\}$. For this, it suffices to show that
\begin{enumerate}
\item\label{yks} $t_i^{\M^*}\langle s' \rangle = s(y_i)$, for $1 \leq i \leq n$,
\item\label{kaks} $u_{ij}^{\M^*} \left< s' \right > = s(z_{ij})$, for $1 \leq i \leq n$ and $1 \leq j \leq c_i$.
\end{enumerate}
Item \ref{yks} is the definition stated in \eqref{pla}. For item \ref{kaks}, let $1 \leq i \leq n$ and $1 \leq j \leq c_i$. Then $u_{ij}$ is $f_i(t_{l_{1}}, \ldots ,t_{l_{k}})$ and
\begin{equation}\label{tästä}u_{ij}^{\M^*}\langle s'\rangle=f_i^{\M^*}(t_{l_{1}}^{\M^*}\left <s'\right >, \ldots ,t_{l_{k}}^{\M^*}\left <s'\right >)=s''(y_i),
\end{equation}
for a $s''\in X$ such that
$$s''(x_1, \ldots ,x_k)=(t_{l_{1}}^{\M^*}\left <s'\right >, \ldots ,t_{l_{k}}^{\M^*}\left <s'\right >).$$
By item \ref{yks}, we obtain that 
$$(t_{l_{1}}^{\M^*}\left <s'\right >, \ldots ,t_{l_{k}}^{\M^*}\left <s'\right >)=  s(y_{l_{1}}, \ldots ,y_{l_{k}}  )$$
when
$$s''(x_1, \ldots ,x_k)=s(y_{l_{1}}, \ldots ,y_{l_{k}}  ).$$
By the assumption we know that
\begin{equation}\label{tälle}
\M \models_X y_{l_{1}} \ldots y_{l_{k}} z_{ij} \subseteq x_1 \ldots x_k y_i.
\end{equation}
Since $s''$ is the only possible witness for \eqref{tälle} with respect to $s$, we conclude that $s''(y_i)=s(z_{ij})$. From this and \eqref{tästä} the equality in item \ref{kaks} follows. This concludes the \textit{if-part} and thus the proof.

\end{proof}
By \cite{cook72}, $\ESOfvar{k} < \ESOfvar{(k+1)}$ follows for any vocabulary (see also Corollary 2.21 in \cite{durand11}). Hence, and by Theorems \ref{t2}, \ref{leqind} and \ref{leqinc}, we obtain the following expressivity hierarchies.
\begin{cor}
For any vocabulary, the following inequalities hold.
\begin{itemize}
\item $\incforall{k} \leq \ESOfvar{k} < \ESOfvar{(k+1)} \leq \incforall{(k+1)}$,
\item $\indforall{k} \leq \ESOfvar{(k+1)} < \ESOfvar{(k+2)} \leq \indforall{(2k+4)}$.
\end{itemize}
\end{cor}


\section{Conclusion}
In this article we have studied the expressive power of fragments of inclusion and independence logic under the strict semantics. Our main result gives an exact characterization of the expressive power of the fragments $\incforall{k}$. On the other hand, determining the exact relationship between $\indforall{k}$and $\ESOfvar{k}$ remains an open problem.
\section*{Acknowledgement}
The authors would like to express their gratitude to Arnaud Durand for valuable comments and ideas. The authors were supported by grant 264917 of the Academy of Finland.

\bibliography{biblio}
\bibliographystyle{plain}
\end{document}